\documentclass[12pt]{amsart}
\usepackage[all]{xy} % for complicated commutative diagrams
\DeclareFontEncoding{OT2}{}{} % to enable usage of cyrillic fonts

\def\bnpx{B^{\e\prime}_{\e n}(x)}
\def\bnx{B_{\e n}(x)}
\def\iv{{\rm{Div}}\e}
\def\cok{{\rm{Coker}}\,}

\def\img{{\rm{Im}}\,}
\def\sh2{{\s H}^{\e 2}(\bq/\bz(2))}

\def\chn{CH^{n}(X)}

\def\kbix{\overline{K}_{i}(x)}
\def\kblx{\overline{K}_{\be 1}\lbe(x)}
\def\chnb{CH^{n}\be\big(\e\overline{X}\e\big)}
\def\br{{\rm{Br}}\e}
\def\s{\mathcal }
\def\kn{K_{n}}
\def\dn{\partial^{\, n-1}}
\def\dnx{\partial^{\, n-1}_{x}}
\def\pic{{\rm{Pic}}\e}
\def\krn{{\rm{Ker}}\,}

\def\bz{{\mathbb Z}\,}

\def\bq{{\mathbb Q}}

\def\spec{{\rm{Spec}}\,}
\def\pdiv{\text{$p\kern 0.1em$-div}}
\def\vbx{V(\overline{x})}
\def\be{\kern -.1em}
\def\lbe{\kern -.05em}
\def\s{\mathcal }
\def\ra{\rightarrow}
\def\e{\kern 0.08em}
\def\le{\kern 0.04em}
\def\ng{\kern -0.04em}

\def\g{\varGamma}
\def\Bnx{B_{n}(X)}
\def\kb{\overline{k}}

\def\yb{\overline{Y}}

\def\xb{\overline{X}}

\newtheorem{theorem}{Main Theorem\!\!}

\newtheorem{lemma}{Lemma}[section]

\newtheorem{corollary}[lemma]{Corollary}
\newtheorem{proposition}[lemma]{Proposition}
\theoremstyle{definition}
\newtheorem{teorema}[lemma]{Theorem}

\theoremstyle{remark}
\newtheorem{remark}[lemma]{Remark}

\title[The generalized Picard-Brauer exact sequence]{A generalization of the
Picard-Brauer exact sequence}


\author{Cristian D. Gonz\'alez-Avil\'es}
\address{Departamento de Matem\'aticas, Universidad de La Serena, La Serena,
Chile}
\email{cgonzalez@userena.cl}


\thanks{The author is partially supported by Fondecyt grant
1080025}

\begin{document}

\begin{abstract} We extend an argument of S.Lichtenbaum involving
codimension one cycles to higher codimensions and obtain a
generalization of the well-known Picard-Brauer exact sequence for a
smooth variety $X$. The resulting exact sequence connects the
codimension $n$ Chow group of $X$ with a certain ``Brauer-like"
group.
\end{abstract}

\maketitle

\section{Introduction.}

Let $k$ be a field and let $X$ be a geometrically integral algebraic
$k$-scheme. We write $\kb$ for a fixed separable algebraic closure
of $k$ and set $\g=\text{Gal}\big(\e\kb/k\big)$. The $\kb$-scheme
$X\times_{\e\spec k}\spec \kb$ will be denoted by $\xb$. Let
$\kb[X]^{*}=H_{\text{\'et}}^{\e 0}(\xb,{\Bbb G}_{m})$ and
$\br^{\prime}X=H_{\text{\'et}}^{\e 2}(X,{\Bbb G}_{m})$ be,
respectively, the group of invertible regular functions on $\xb$ and
the cohomological Brauer group of $X$. The exact sequence mentioned
in the title is the familiar exact sequence

\begin{equation}
\begin{array}{rcl}
0&\ra & H^{\e 1}(k,\kb[X]^{*})\ra\pic X \ra
\big(\e\pic\xb\e\big)^{\lbe\g}\ra H^{\e 2}(k,\kb[X]^{*})
\ra\br_{\be 1}^{\e\prime}X\\
&\ra & H^{\e 1}\big(k,\pic\xb\e\big)\ra H^{\e 3}(k,\kb[X]^{*})
\end{array}
\end{equation}
where $H^{\e i}(k,-)=H^{\e i}(\g,-)$ and
$\br^{\e\prime}_{1}X=\krn\big(\br^{\e\prime}X\ra
\br^{\e\prime}\e\xb\,\big)$. This sequence may be obtained from the
exact sequence of terms of low degree belonging to the
Hochschild-Serre spectral sequence
$$
H^{\e r}\be\big(k,H_{\text{\'et}}^{\e s}\big(\e\xb,{\Bbb
G}_{m}\big)\big)\implies H_{\text{\'et}}^{\e r+s}(X,{\Bbb G}_{m}).
$$
When $X$ is {\it smooth} (which we assume from now on), there exists
an alternative derivation of (1) which makes use of the following
(no less familiar) exact sequence:
\begin{equation}
0\ra\kb[X]^{*}\ra\kb(X)^{*}\ra\text{Div}\e\xb\ra \pic\xb\ra 0
\end{equation}
where $\kb(X)^{*}$ (resp. $\text{Div}\e\xb\e$) is the group of
invertible rational functions (resp. Cartier divisors) on $\xb$.
This approach, seemingly first used by S.Lichtenbaum in [4] and then
reconsidered by Yu.Manin [5, p.403\,], consists in splitting (2)
into two short exact sequences of $\g$-modules and then taking
$\g$-cohomology of these sequences. The resulting long
$\g$-cohomology sequences are then appropriately combined to produce
(1). This paper is a generalization of this idea. The key
observation to make is that (2) may be seen as arising from the
Gersten-Quillen complex corresponding to the Zariski sheaf
${\mathcal K}_{\e 1,\e\xb}$, which is the sheaf on $\xb$ associated
to the presheaf $U\mapsto K_{1}(U)=H^{\e 0}(U,{\mathcal
O}_{U})^{*}$. In Section 2 we work with the Gersten-Quillen complex
corresponding to the Zariski sheaf ${\mathcal K}_{\e n,\e\xb}$
associated to the presheaf $U\mapsto\kn(U)$, where $\kn$ is
Quillen's $n$-th $K$-functor ($1\leq n\leq d=\text{dim}(X)$), and
obtain the following result. Let $\dn\colon\bigoplus_{y\in \xb^{\e
n-1}}\kb(y)^{*}\ra Z^{\e n}\lbe\big(\e\xb\e\big)$ be the ``sum of
divisors" map and let $\Bnx$ be the kernel of the induced map
$$
H^{\e 2}\!\left(k,\textstyle\bigoplus_{y\in \xb^{\e
n-1}}\kb(y)^{*}\right)\ra H^{\e 2}\big(k,Z^{\e
n}\lbe\big(\e\xb\e\big)\big).
$$

\begin{theorem} Let $X$ be a smooth, geometrically integral,
algebraic $k$-scheme. Then there exists a natural exact sequence
$$\begin{array}{rcl}
0 &\ra& H^{\e 1}(k,\krn\dn)\ra\chn\ra\chnb^{\g}\!\ra H^{\e
2}(k,\krn\dn)\\
&\ra&\Bnx\ra H^{\e 1}\big(k,\chnb\big)\ra H^{\e 3}(k,\krn\dn)\,.
\end{array}
$$
\end{theorem}

\medskip

The case $n=1$ of the theorem is precisely the exact sequence (1).

\medskip

In Section 4, which concludes the paper, we show that the group
$\Bnx$ in the exact sequence of the theorem is ``Brauer-like", in
the sense that it contains a copy of $\br_{1} Y=\krn\be\big[\br
Y\ra\br\yb\,\big]$ for every smooth closed integral subscheme
$Y\subset X$ of codimension $n-1$.

\section{Preliminaries}

We keep the notations of the Introduction. In particular, $X$ is a
smooth, geometrically integral algebraic $k$-scheme of dimension $d$
and $n$ denotes a fixed integer such that $1\leq n\leq d$.

There exists a natural bijection between the set of schematic points
of $X$ and the set of closed integral subschemes of $X$. This is
defined by associating to a point $x\in X$ the schematic closure
$V(x)$ of $x$ in $X$. The codimension (resp. dimension) of $x$ is by
definition the codimension (resp. dimension) of $V(x)$. The set of
points of $X$ of codimension (resp. dimension) $i$ will be denoted
by $X^{i}$ (resp. $X_{i}$), and $\eta$ (resp. $\overline{\eta}$)
will denote the generic point of $X$ (resp. $\xb$). If $x\neq\eta$,
the function field of $V(x)$ will be denoted by $k(x)$. We use the
standard notation $k(X)$ for the function field of $X=V\be(\eta)$.
For each $x\in X$, $i_{x}$ will denote the canonical map
$\text{Spec}\e\e k(x)\ra X$. The function field of $\xb$ will be
denoted by $\kb(X)$. For simplicity, we will write $\vbx$ for
$V(x)\times_{\spec k}\spec \kb$.

Since $\xb$ is regular [3, 6.7.4], the sheaf ${\s K}_{\e n,\e\xb}$
admits the following flasque resolution, known as the
Gersten-Quillen resolution (see [7, p.72]):

$$\begin{array}{rcl}
0\ra {\s K}_{\e n,\e\xb} &\ra(\e i_{\e\overline{\eta}}\e)_{*}\e
K_{n}\kb(X) \ra\displaystyle\bigoplus_{y\in\xb^{\e 1}}(\e i_{\e
y})_{*}K_{n-1}\e\kb(y)\ra
\dots\\
&\ra\!\!\displaystyle\bigoplus_{y\in\xb^{\e n-1}}(\e i_{\e
y})_{*}\e\kb(y)^{*} \ra \displaystyle\bigoplus_{y\in\xb^{\e n}}(\e
i_{\e y})_{*}\e\bz\e\ra 0
\end{array}
$$
where, for $y\in\xb^{\, i}$, $K_{n-i}\,\kb(y)$ is regarded as a
constant sheaf on $\kb(y)$. It follows that the groups $H^{\e
i}\lbe\big(\e\xb,{\s K}_{\e n,\xb}\e\big)=H^{\e i}\lbe\big(\e\xb,{\s
K}_{\e n}\e\big)$ are the cohomology groups of the complex
\begin{equation}
\kn\kb(X)\overset{\partial^{\e
0}}\longrightarrow\displaystyle\bigoplus_{y\in\xb^{\e 1}}\be
K_{n-1}\e\kb(y)\overset{\partial^{\e 1}}\longrightarrow\dots
\overset{\partial^{\e n\be-\be
2}}\longrightarrow\displaystyle\bigoplus_{y\in\xb^{\e
n-1}}\ng\be\kb(y)^{*}\overset{\partial^{\e n\be-\be
1}}\longrightarrow\displaystyle\bigoplus_{y\in\xb^{\e n}}\be\Bbb
Z\,.
\end{equation}
Now, if $q\colon\xb\ra X$ is the canonical morphism and $x\in X$, we
write $\xb^{\, n-i}_{\be x}$ for the set of points $y\in \xb^{\e
n-i}$ such that $q(y)=x$. For $i=1,2,\dots, n\be -\be 1$ and $\,
x\in X^{n-i}$, set
$$
\kbix=\bigoplus_{y\e\in\e\xb^{\e n-i}_{\be x}} K_{i}\e\kb(y).
$$
Further, write $Z^{\e n}\be\big(\e\xb\e\big)$ for the group of
codimension $n$ cycles on $\xb$, i.e.,
$$
Z^{\e n}\be\big(\e\xb\e\big)=\displaystyle\bigoplus_{y\in\xb^{\e
n}}\Bbb Z\,.
$$
Then (3) may be written as
\begin{equation}
\kn\kb(X)\overset{\partial^{\e 0}}\longrightarrow\bigoplus_{x\in
X^{1}}\be \overline{K}_{n-1}(x)\overset{\partial^{\e
1}}\longrightarrow\dots \overset{\partial^{\e n\be-\be
2}}\longrightarrow\ng\ng\bigoplus_{x\in X^{n-1}}\ng\be
\kblx\overset{\partial^{\e n\be-\be 1}}\longrightarrow Z^{\e
n}\be\big(\e\xb\e\big)\,.
\end{equation}
The differential $\dn$ equals $\sum_{\e x\in X^{\lbe n-1}}\dnx\,$,
where, for each $x\in X^{n-1}$,
$$
\dnx\colon\kblx=\displaystyle\bigoplus_{y\in\xb^{\e n-1}_{\be x}}
\kb(y)^{*} \ra Z^{\e n}\be\big(\e\xb\e\big)
$$
is the sum of the divisor maps
$$
\text{div}_{\be y}\colon\kb(y)^{*}\ra Z^{\e n}\be\big(\e\xb\e\big).
$$
For definition of the latter, see [7, p.72]. We note that each of
the maps $\text{div}_{ y}$ factors through $Z^{\e 1}(\e V(y))$,
whence each $\dnx$ factors through $Z^{\e 1}(\e\vbx)$.

We will write  $CH^{\e n}\lbe(X)$ for the Chow group of codimension
$n$ cycles on $X$ modulo rational equivalence. Then $H^{\e n}(X,\s
K_{n})=CH^{\e n}\lbe(X)$ (``Bloch's formula").

\section{Proof of the main theorem}

The complex (4) induces the following short exact sequences of
$\g$-modules:
\begin{equation}
 0\ra \img\dn\ra Z^{\e
n}\be\big(\e\xb\e\big)\ra\chnb\ra 0
\end{equation}
and
\begin{equation} 0\ra\krn\dn\ra\displaystyle\bigoplus_{x\in
X^{n-1}} \kblx\ra\img\dn\ra  0.
\end{equation}
Observe that the natural morphism $q\colon\xb\ra X$ induces a
homomorphism $\chn\ra\chnb^{\g}$.

\begin{lemma} There exist canonical isomorphisms
$$
\begin{array}{rcl}
\krn\!\be\left[\chn\ra\chnb^{\g}\e\right]&=& H^{\e
1}\be\big(k,\krn\dn\e\big)\\\\
\cok\!\be\left[\chn\ra\chnb^{\g}\e\right]&=& H^{\e
1}\be\big(k,\img\dn\e\big)
\end{array}
$$
and a canonical exact sequence
$$
0\ra H^{\e 1}\be\big(k,\chnb\big)\ra H^{\e
2}\be\big(k,\img\dn\e\big)\ra H^{\e 2}\big(k,Z^{\e n}
\be\big(\e\xb\e\big)\big).
$$
\end{lemma}
\begin{proof} This follows by taking $\g$-cohomology of (5), using the fact
that $Z^{\e n}(\xb)$ is a permutation $\g$-module and arguing as in
[1, proof of Proposition 3.6] to establish the first isomorphism.
\end{proof}

\begin{lemma} The exact sequence (6) induces an exact sequence
$$
\begin{array}{rcl}
 0&\ra & H^{\e 1}\be\big(k,\img\dn\big)
\ra H^{\e 2}\be\big(k,\krn\dn\big)\ra
\displaystyle\bigoplus_{x\in X^{\lbe n-1}} H^{\e 2}\big(k,\kblx\big)\\
&\ra & H^{\e 2}\be\big(k,\img\dn\big)\ra H^{\e
3}\be\big(k,\krn\dn\big)\,.
\end{array}
$$
\end{lemma}
\begin{proof} By Shapiro's Lemma, for each $x\in X^{n-1}$ there exists
a (non-canonical) isomorphism
$$
H^{*}(k,\kblx)\simeq
H^{*}\be\big(\e\text{Gal}\big(\e\kb(y)/k(x)\big),\kb(y)^{*}\big)
$$
where, on the right, we have chosen a point $y\in\xb^{\, n-1}$ such
that $q(y)=x$. The result now follows by taking $\g$-cohomology of
(6), using Hilbert's Theorem 90.
\end{proof}

Combining Lemmas 3.1 and 3.2, we obtain

\begin{proposition} There exists a canonical exact sequence
$$\begin{array}{rcl}
0&\ra & H^{\e 1}\be\big(k,\krn\dn\e\big)\ra\chn\ra\chnb^{\g}\\
&\ra & H^{\e
2}\be\big(k,\krn\dn\e\big)\ra\displaystyle\bigoplus_{x\in X^{n-1}}
H^{\e 2}(k,\kblx).\qed
\end{array}
$$
\end{proposition}

Now define
\begin{equation}
\Bnx=\krn\!\be\left[\e H^{\e 2}\!\left(k,\textstyle\bigoplus_{x\in
X^{n-1}}\kblx\right)\ra H^{\e 2}\big(k,Z^{\e
n}\be\big(\e\xb\e\big)\big)\right],
\end{equation}
where the map involved is induced by $\dn$. Since the composite
$$
\krn\dn\ra\bigoplus_{x\in X^{n\be -\be 1}}
\kblx\overset{\dn}\longrightarrow Z^{\e n}\be\big(\e\xb\e\big)
$$
is zero, the natural map $H^{\e 2}(k,\krn\dn)\ra\be\bigoplus_{x\in
X^{n-1}}\be H^{\e 2}(k,\kblx)$ factors through $\Bnx$. Thus
Proposition 3.3 yields a natural exact sequence
\begin{equation}
\begin{array}{rcl}
0&\ra & H^{\e 1}\be\big(k,\krn\dn\e\big)\ra\chn\ra\chnb^{\g}\\
&\ra & H^{\e 2}\be\big(k,\krn\dn\e\big)\ra\Bnx.
\end{array}
\end{equation}
We will now extend the above exact sequence by defining a map
$\Bnx\ra H^{\e 1}(k,\chnb)$ whose kernel is exactly the image of the
map $H^{\e 2}(k,\krn\dn)\ra\Bnx$ appearing in (8).

It is not difficult to check that the map
$$
\displaystyle\bigoplus_{x\in X^{n-1}} H^{\e 2}\big(k,\kblx\big)\ra
H^{\e 2}\be\big(k,\img\dn\big)
$$
intervening in the exact sequence of Lemma 3.2 maps $\Bnx$ into the
kernel of the map $H^{\e 2}(k,\img\dn)\ra H^{\e 2}\be\big(k,Z^{\e n}
\be\big(\e\xb\e\big)\big)$. The latter is naturally isomorphic to
$H^{\e 1}\be\big(k,\chnb\big)$ (see Lemma 3.1). Thus there exists a
canonical map $\Bnx\ra H^{\e 1}\be\big(k,\chnb\big)$. Again, it is
not difficult to check that the kernel of the map just defined is
exactly the image of the map $H^{\e 2}(k,\krn\dn)\ra\Bnx$ appearing
in (8). Thus we obtain a natural exact sequence
$$
\begin{array}{rcl}
0 &\ra& H^{\e 1}(k,\krn\dn)\ra\chn\ra\chnb^{\g}\ra H^{\e 2}
(k,\krn\dn)\\
&\ra &\Bnx\ra H^{\e 1}(k,\chnb)\,.
\end{array}
$$
Finally, the homomorphisms $H^{\e 1}\be(k,\chnb)\be\ra\be H^{\e
2}\be(k,\img\dn)$ and $H^{\e 2}\be(k,\img\dn)\be\ra\be H^{\e
3}\be(k,\krn\dn)$ from Lemmas 3.1 and 3.2 induce a map $H^{\e
1}(k,\chnb)\ra H^{\e 3}(k,\krn\dn)$ whose kernel is exactly the
image of the map $\Bnx\ra H^{\e 1}(k,\chnb)$ defined above. Thus the
following holds.

\begin{teorema} Let $X$ be a smooth $k$-variety. Then there exists
a natural exact sequence
$$\begin{array}{rcl}
0 &\ra & H^{\e 1}(k,\krn\dn)\ra\chn\ra\chnb^{\g}\ra
H^{\e 2}(k,\krn\dn)\\
&\ra &\Bnx\ra H^{\e 1}\be\big(k,\chnb\big)\ra H^{\e 3}(k,\krn\dn),
\end{array}
$$
where $\Bnx$ is the group (7).
\end{teorema}

\begin{remark} When $n=1$, there are natural isomorphisms
$\chn=\pic X$ and $\chnb=\pic\xb\,$ [3, 21.6.10 and 21.11.1].
Further, $X^{ n-1}=\{\eta\}\e$, $\,\dn=\partial^{\e
n-1}_{\eta}\colon\overline{K}_{\be 1}(\eta)=\kb(X)^{*}\ra\iv\xb$ is
the usual divisor map (whose kernel equals $H^{0}(\xb,\Bbb
G_{m})\overset{\text{def.}}{=}\kb[X]^{*}$) and
$$
\Bnx=B_{\e 1}(X)=\krn\!\be\left[\e H^{\e 2}(k,\kb(X)^{*})\ra H^{\e
2}\be\big(k,\iv\xb\,\big)\right]= \br_{\be 1}X,
$$
where $\br_{\be 1}X=\krn\!\left(\e\br X\ra\br\xb\,\right)$ (see the
next section). Thus the exact sequence of the theorem is indeed a
generalization of (1).
\end{remark}

\section{The group $\Bnx$}

In this Section we show that the group $\Bnx$ appearing in the exact
sequence of Theorem 3.4 contains a copy of $\br_{1}
Y=\krn\be\big(\e\br Y\ra\br\yb\,\big)$ for {\it every}\, smooth
closed integral subscheme $Y\subset X$ of codimension $n-1$.

\smallskip

Recall that $\dn=\sum_{\e x\in X^{\lbe n-1}}\dnx\,$, where, for each
$x\in X^{n-1}$,
$$
\dnx\colon\kblx=\displaystyle\bigoplus_{y\in\xb^{\e n-1}_{\be x}}
\kb(y)^{*} \ra Z^{\e 1}\be\big(\vbx\big)
$$
is the sum of divisors map. For each $x\in X^{n-1}$, set
$$
\bnx=\krn\!\left[\,H^{\e 2}(k,\kblx)\ra H^{\e
2}(k,Z^{\e1}(\vbx\e))\e\right],
$$
where the map involved is induced by $\dnx$, and let
$$
\Sigma\,\,\colon\bigoplus_{x\in X^{n-1}} H^{\e 2}(k,Z^{\e
1}(\e\vbx))\ra H^{\e 2}(k,Z^{\e n}(\xb))
$$
be the natural map $(\xi_{x})\mapsto\sum c_{x}(\xi_{x})$, where
$c_{x}\colon H^{\e 2}(k,Z^{\e 1}(\e\vbx))\ra H^{\e 2}(k,Z^{\e
n}(\xb))$ is induced by the inclusion $Z^{\e 1}(\e\vbx)\subset Z^{\e
n}(\xb)$. Then there exists a canonical exact sequence
$$
0 \ra \displaystyle\bigoplus_{x\in X^{n-1}}\bnx\ra\Bnx\ra
\krn\Sigma\,\,.
$$
We will relate the groups $\bnx$ to more familiar objects.

\medskip

Fix $x\in X^{ n\be -\be 1}$ and set $Y=V(x)$. Then $Y$ is a
geometrically reduced algebraic $k$-scheme [3, 4.6.4\,]. Further,
the map $\kblx\ra Z^{\e 1}\be\big(\e\yb\e\big)$ factors through
$\iv\yb$, the group of Cartier divisors on $\yb$. Consider
\begin{equation}
\bnpx=\krn\ng\left[H^{\e2}(k,\kblx)\ra
H^{\e2}\be\big(k,\iv\yb\e\big)\e\right]\subset\bnx.
\end{equation}
Let ${\s R}_{\,\yb}^{*}$ denote the \'etale sheaf of invertible
rational functions on $\yb$. Note that $\kblx\be=\be H^{\e
0}(\yb,{\s R}_{\,\yb}^{*})$. Now, since $\yb$ is reduced, there
exists an exact sequence of \'etale sheaves
$$0\ra\Bbb G_{m,\yb}\ra{\s R}_{\,\yb}^{*}\ra{\s Div}_{\e\yb}\ra 0,$$
where ${\s Div}_{\yb}$ is the sheaf of Cartier divisors on $\yb$ [3,
20.1.4 and 20.2.13\,]. This exact sequence gives rise to an exact
sequence of \'etale cohomology groups
\begin{equation}
0\ra H_{\text{\'et}}^{\e 1}\be\big(\e\yb,{\s
Div}_{\e\yb}\e\big)\ra\br'\,\yb\ra H_{\text{\'et}}^{\e
2}\be\big(\e\yb,{\s R}_{\,\yb}^{*}\big) \ra H_{\text{\'et}}^{\e
2}\be\big(\e\yb,{\s Div}_{\e\yb}\big)
\end{equation}
where $\br'\,\yb=H_{\text{\'et}}^{\e 2}\be\big(\e\yb, \Bbb
G_{m}\big)$ is the cohomological Brauer group of $\yb$ [2, II,
p.73]. Similarly, there exists an exact sequence
\begin{equation}
0\ra H_{\text{\'et}}^{\e 1}(Y,{\s Div}_{Y})\ra\br'\,Y\ra
H_{\text{\'et}}^{\e 2}(Y,{\s R}_{\,Y}^{*}) \ra H_{\text{\'et}}^{\e
2}(Y,{\s Div}_{Y}).
\end{equation}
We will regard $H_{\text{\'et}}^{\e 1}\be\big(\e\yb,{\s
Div}_{\e\yb}\e\big)$ (resp. $H_{\text{\'et}}^{\e 1}(Y,{\s
Div}_{Y})$) as a subgroup of $\br'\,\yb$ (resp. $\br'\,Y$).

 Now the
exact sequence of terms of low degree
$$
0\ra E_{2}^{\e 1,0}\ra E^{\e 1}\ra E_{2}^{\e 0,1}\ra E_{2}^{\e 2,0}
\ra\krn(E^{\e 2}\ra E_{2}^{\e 0,2})\ra E_{2}^{\e 1,1}\ra E_{2}^{\e
3,0}
$$
belonging to the Hochschild-Serre spectral sequence
$$
E_{2}^{\e p,q}=H^{\e p}\be\big(k,H^{\e
q}_{\text{\'et}}\be\big(\e\yb,{\s R}_{\,\yb}^{*}\big)\e\big)\implies
H_{\text{\'et}}^{\e p+q} (Y,{\s R}_{\, Y}^{*})
$$
yields, using [2, II, Lemma 1.6, p.72], an exact sequence
\begin{equation}
0\ra H^{\e 2}(k,\kblx)\ra H_{\text{\'et}}^{\e 2}(Y,{\s R}_{\,Y}^{*})
\ra H_{\text{\'et}}^{\e 2}\be\big(\e\yb,{\s R}_{\,\yb}^{*}\big).
\end{equation}
Similarly, the spectral sequence
$$
H^{\e p}\be\big(k,H_{\text{\'et}}^{\e q}\be\big(\e\yb,{\s
Div}_{\e\yb}\big)\big) \implies H_{\text{\'et}}^{\e p+q}(Y,{\s
Div}_{Y})
$$
yields a complex
\begin{equation}
\begin{array}{rcl}
0 &\ra& H^{\e 1}(k,\iv\yb\e)\ra H_{\text{\'et}}^{\e 1}(Y,{\s
Div}_{Y}) \overset{\psi}\longrightarrow H_{\text{\'et}}^{\e
1}\be\big(\e\yb,{\s Div}_{\e\yb}\big)^{\g}\\
&\overset{\varphi}\longrightarrow & H^{\e
2}\be\big(k,\iv\yb\e\big)\ra H_{\text{\'et}}^{\e 2}(Y, {\s
Div}_{Y})\ra H_{\text{\'et}}^{\e 2}\be\big(\e\yb,{\s
Div}_{\e\yb}\big)
\end{array}
\end{equation}
which is exact except perhaps at $H_{\text{\'et}}^{\e 2}(Y, {\s
Div}_{Y})$. The map labeled $\psi$ in (13) is induced by the
canonical morphism $\yb\ra Y$ , while the map $\varphi$ is the
differential $d_{2}^{\, 0,1}$ coming from the spectral sequence (see
[6, II.4, pp.39-52]). Now we have a commutative diagram

\begin{equation}
\xymatrix{0\ar[r] & H^{\e 2}(k,\kblx)\ar[d]\ar[r]&
H_{\text{\'et}}^{\e 2} (Y,{\s R}_{\,Y}^{*})\ar[d]\ar[r] &
H_{\text{\'et}}^{\e 2}\be\big(\e\yb,{\s
R}_{\,\yb}^{*}\big)\ar[d]\\
0\ar[r] & H^{\e 2}\be\big(\e k,\iv\yb\e\big)/\e\img\varphi \ar[r]&
H_{\text{\'et}}^{\e 2}(Y, {\s Div}_{\e Y})\ar[r] &
H_{\text{\'et}}^{\e 2}\be\big(\e\yb,{\s Div}_{\e\yb}\big).\\
}
\end{equation}
in which the top row is the exact sequence (12), the bottom row
(which is only a complex) is derived from (13), and the middle and
right-hand vertical maps are the maps in (11) and (10),
respectively. Set
$$
\widehat{\text{Br}_{1}^{\,\prime}}\e Y=\krn\be\Bigg[\e\br'\e
Y/H_{\text{\'et}}^{\e 1}(Y,{\s Div}_{Y})\ra
\br'\,\yb/H_{\text{\'et}}^{\e 1}\be\big(\e\yb,{\s
Div}_{\e\yb}\big)\e\Bigg].
$$
Then the above diagram yields a natural isomorphism
\begin{equation}
\widehat{\text{Br}_{1}^{\,\prime}}\e Y= \krn\!\be\left[\e H^{\e
2}(k,\kblx)\ra H^{\e 2}\be\big(\e k,\iv\yb\,\big)/\e\img
\varphi\,\right].
\end{equation}
(Note: only the exactness of the top row of (14) is needed to obtain
the above isomorphism.) On the other hand, there exists an obvious
exact sequence
$$
0\ra\bnpx\ra\krn\!\be\left[\e H^{\e 2}(k,\kblx)\ra H^{\e
2}\be\big(\e k,\iv\yb\,\big)/\e\img\varphi\,\right]\ra\img\varphi,
$$
where $\bnpx$ is the group (9). Using (15) and the fact that
$\img\varphi$ is naturally isomorphic to $\cok\psi$, where $\psi$ is
the map appearing in (13), we conclude that there exists a natural
exact sequence
\begin{equation}
0\ra\bnpx\ra\widehat{\text{Br}_{1}^{\,\prime}}\e Y
\overset{h}\longrightarrow\cok\psi\,.
\end{equation}
The map labeled $h$ in the above exact sequence can be briefly
described as $``\e\varphi^{-1}\circ h^{\e 2}(\text{div})\e\circ\e
u^{-1}"$, where $u\colon H^{\e 2}(k,\kblx)\ra H^{\e
2}_{\text{\'et}}(Y,{\s R}_{\,Y}^{*})$ is the map intervening in (14)
and $h^{\e 2}(\text{div})\colon H^{\e 2}(k,\kblx)\ra H^{\e
2}\be\big(\e k,\iv\yb\,\big)$ is induced by
$\text{div}\colon\kblx\ra \iv\yb$. Next, set
$$
\br_{1}^{\e\prime}\e
Y=\krn\!\left[\e\br^{\e\prime}Y\ra\br^{\e\prime}\e \yb\,\right].
$$
There exists a natural exact commutative diagram
$$
\xymatrix{0\ar[r] & H^{\e 1}_{\text{\'et}}(Y,{\s
Div}_{Y})\ar[d]^{\psi}\ar[r]& \br_{\be 1}^{\e\prime}\e
Y\ar[d]\ar@{->>}[r] & \br_{\be 1}^{\e\prime}\e Y/\e
H^{\e 1}_{\text{\'et}}(Y,{\s Div}_{Y})\ar[d]\\
0\ar[r] & H^{\e 1}_{\text{\'et}}\be\big(\e\yb,{\s
Div}_{\e\yb}\big)^{\g} \ar[r]&  \left(\e\br^{\prime}\e\yb\,\right
)^{\g}\ar[r] & \left(\e\br^{\prime}\e\yb/\e H^{\e
1}_{\text{\'et}}\be\big(\e\yb,{\s Div}_{\e\yb}\big)\e\right)^{\g}.}
$$
An application of the snake lemma to the above diagram yields a
natural exact sequence
\begin{equation}
0\ra H^{\e 1}(k,\iv\yb\e)\ra \br_{\be 1}'\e Y\ra
\widehat{\text{Br}_{1}^{\,\prime}}\e Y
\overset{\delta}\longrightarrow\cok\psi\,.
\end{equation}
Now using the explicit description of the map $\delta$ [8, Lemma
1.3.2, p.11] together with the description of the map
$\varphi=d_{2}^{\, 0,1}$ from [6, \S II.4\,], it can be shown (with
some work) that the maps $h$ in (16) and $\delta$ in (17) {\it are
the same}. Thus we obtain

\begin{proposition} There exists a canonical isomorphism
$$
\bnpx=\br_{1}^{\e\prime}\e Y/\e H^{\e 1}\be\big(k,\iv\yb\e\big).
$$
\end{proposition}

\begin{corollary} Let $x\in X^{ n-1}$ be such that
$\yb=\vbx$ is {\rm{locally factorial}} (this holds, for example, if
$Y=V(x)$ is regular). Then there exists a canonical isomorphism
$$
\bnx=\br_{\be 1}^{\e\prime}\e Y.
$$
\end{corollary}
\begin{proof} The hypothesis implies that $\iv\yb=Z^{\e 1}\be\big
(\e\yb\e\big)$ [3, 21.6.9\,], so $\bnx=\bnpx$. On the other hand,
since $Z^{\e 1}\be\big (\e\yb\e\big)$ is a permutation $\g$-module,
$H^{\e 1}\be\big(k,\iv\yb\e\big)=H^{\e 1}\be\big(k,Z^{\e 1}\be\big
(\e\yb\e\big)\big)=0$. The result is now immediate from the
proposition.
\end{proof}


\begin{thebibliography}{}


\bibitem[1]{1} Colliot-Th\'el\`ene, J.-L. and Raskind, W.:
\emph{ ${\s K}_{2}$-Cohomology and the second Chow group} Math. Ann.
{\bf{270}}, pp.165-199 (1985).

\bibitem[2]{2} Grothendieck, A.:\emph{ Le Groupe de Brauer
I-III.} In: Dix Expos\'es sur la Cohomologie des Sch\'emas.
North-Holland, Amsterdam, pp.46-188 (1968).

\bibitem[3]{3} Grothendieck, A. and Dieudonn\'e, J.:\emph{
El\'ements de G\'eom\'etrie Alg\'ebrique IV.} Publ. Math. IHES
{\bf{20,24,28,32}}, 1960-1967.

\bibitem[4]{4} Lichtenbaum, S.:\emph{ Duality theorems for curves
over $P$-adic fields.} Invent. Math. {\bf{7}}, pp. 120-126 (1969).

\bibitem[5]{5} Manin, Yu.:\emph{ Le groupe de Brauer-Grothendieck en
g\'eom\'etrie diophantienne.} In: Actes du Congr\`es Intern. Math.
Nice I, pp. 401-411 (1970).


\bibitem[6]{6} Shatz, S.:\emph{ Profinite groups, Arithmetic, and
Geometry.} Ann. of Math. Studies {\bf{67}}, Princeton Univ. Press
(1972).

\bibitem[7]{7} Srinivas, V.:\emph{ Algebraic
$K$-theory (2nd. Edition)} Progress in Math. {\bf{90}}, Birh\"auser,
Boston, 1996.


\bibitem[8]{8} Weibel, C.:\emph{ An introduction to homological
algebra. } Cambridge Studies in Advanced Math. {\bf{38}}, Cambridge
University Press, Cambridge, 1994.

\end{thebibliography}
\end{document}